\documentclass{article}
\usepackage{anysize}
\marginsize{1in}{1in}{1in}{1in}
\usepackage{amssymb}
\usepackage{amsmath}
\usepackage{amsthm}
\usepackage{mathtools}
\newtheorem{defn}{Definition}[section]
\newtheorem{thm}{Theorem}[section]
\newtheorem{lemma}{Lemma}[section]

\newtheorem{note}{Note}[section]
\newtheorem{cor}{Corollary}[section]

\theoremstyle{definition}
\numberwithin{equation}{section}
\begin{document}
\title{The number of $k$-tons in the coupon collector problem}
\author{J.C. Saunders}
\date{}
\maketitle
\begin{abstract}
Consider the coupon collector problem where each box of a brand of cereal contains a coupon and there are $n$ different types of coupons. Suppose that the probability of a box containing a coupon of a specific type is  $1/n$ and that we keep buying boxes until we collect at least $m$ coupons of each type. For $k\geq m$ call a certain coupon a $k$-ton if we see it $k$ times by the time we have seen $m$ copies of all of the coupons. Here we determine the asymptotic distribution of the number of $k$-tons after we have collected $m$ copies of each coupon for any $k$ in a restricted range, given any fixed $m$. We also determine the asymptotic joint probability distribution over such values of $k$ and the total number of coupons collected. 
\end{abstract}
\section{Introduction}
Consider the coupon collector problem where each box of a brand of cereal contains a coupon and there are $n$ different types of coupons. Suppose that the probability of a box containing a coupon of a specific type is $\frac{1}{n}$ and that we keep buying boxes until we collect at least one coupon of each type. It is well known, see for example \cite{newman}, that the expected number of boxes we need to buy is $nH_n$, where $H_n$ is the $n$th Harmonic number:
\begin{equation*}
H_n=1+\frac{1}{2}+\frac{1}{3}+\ldots+\frac{1}{n}=\log n+\gamma+\frac{1}{2n}+O\left(\frac{1}{n^2}\right),
\end{equation*} 
where $\gamma$ is the Euler-Mascheroni constant. The expected value of the total number of boxes collected has been extensively studied. For instance, Newman and Shepp \cite{newman} proved that, if you continue to collect boxes until you have at least $m$ coupons of each type, then the expected number of boxes collected is $n\log n+(m-1)n\log\log n+n\cdot C_m+o(n)$, where $C_m$ is a constant depending on $m$. 
\begin{defn}
The Gumbel distribution with parameters $\mu$ and $\beta$ denoted as $\mathrm{Gumbel}(\mu,\beta)$ is defined as the probability distribution with cumulative density function $e^{e^{-(x-\mu)/\beta}}$ for $-\infty<x<\infty$.
\end{defn}
Erd\H{o}s and R\'{e}nyi \cite{erdos} improved the result above and proved that, if $v_m(n)$ is the total number of boxes collected, then $\frac{v_m(n)}{n}-\log n-(m-1)\log\log n\sim\mathrm{Gumbel}(-\log((m-1)!),1)$ as $n\rightarrow\infty$. Baum and Billingsley \cite{baum} studied the probability distribution of the number of boxes needed to be collected to collect $a_n+1$ coupons, where $0\leq a_n<n$. Berenbrink and Sauerwald \cite{berenbrink} examined the expected number of boxes needed to be collected if the types of coupons have different probabilities of being in a box, and Neal \cite{neal} and Doumas and Papanicolaou \cite{doumas} looked at the actual probability distribution of the number of boxes collected if the types of coupons have different probabilities. Doumas and Papanicolaou \cite{doumas2} also examined the case of collecting at least $m$ coupons of each type. Foata and Han \cite{foata} and Myers and Wilf \cite{myers} studied further variations of the problem, where you have multiple collectors.
\newline
\newline
Here we examine the situation where we continue to collect boxes until we have at least $m$ copies of each type of coupons for any fixed $m\in\mathbb{N}$.
\begin{defn}
For all $k\geq m$, call a certain coupon a $k$-ton if we see it $k$ times by the time we have seen $m$ copies of all of the coupons. Let the number of $k$-tons be denoted as $S_k$.
\end{defn}
In the case of $m=1$, Myers and Wilf \cite{myers} determined the expected number of $1$-tons, which they called singletons, to be $H_n$. Penrose \cite{penrose} also studied singletons, but in the case of unequal coupon probabilities and where the number of coupons in total collected is proportional to the number of types of coupons, and noted many applications to their study, such as in data bases, biological particles, and communication channels. Here we extend Myers' and Wilf's result by determining the asymptotic joint probability distribution over values of $k$ in a restricted range and the total number of coupons collected assuming, like Myers and Wilf, equal coupon probabilities. We also determine the asymptotic distribution of the number of $k$-tons after we have collected $m$ copies of each coupon for any such values of $k$, given any fixed $m\in\mathbb{N}$. While making an analogous extension to Penrose's work is certainly interesting, we leave it for a future project.
\begin{thm}\label{thm3}
Fix $m\in\mathbb{N}$ and let $T_m(n)$ be the total number of coupons collected to see at least $m$ copies of each coupon and let $n\rightarrow\infty$. Let $k=o(\log n)$ if $m=1$ and $k=o\left(\frac{\log n}{\log\log n}\right)$ if $m\geq 2$, with $k\geq m$ in either case. Then the joint probability distribution of $\left(\frac{T_m(n)-n\log n-(m-1)n\log\log n}{n},\frac{k!S_k}{(\log n)^{k-m+1}}\right)$ converges to $\left(X,e^{-X}\right)$, where $X\sim\mathrm{Gumbel}(-\log((m-1)!),1)$.
\end{thm}
A simple heuristic argument for Theorem \ref{thm3} may be given as follows. First, by Erd\H{o}s and R\'{e}nyi \cite{erdos}, we know that $\frac{T_m(n)-n\log n-(m-1)n\log\log n}{n}\sim\mathrm{Gumbel}(-\log((m-1)!),1)$. Later on we also see unsurprisingly that the number of times a specific type of coupon is seen is essentially independent of whether or not all types of coupons have been collected at least $m$ times. Also, unsurprisingly, it turns out that the number of times two different types of coupons are seen is asymptotically independent. Since the probability of a type of coupon being seen $k$ times is
\begin{equation*}
\binom{T_m(n)}{k}\left(\frac{1}{n}\right)^{k}\left(1-\frac{1}{n}\right)^{T_m(n)-k}\sim\frac{(n\log n)^ke^{-\log n-(m-1)\log\log n-x}}{k!n^k}=\frac{(\log n)^{k-m+1}}{ne^xk!},
\end{equation*}
where $x=\frac{T_m(n)-n\log n-(m-1)n\log\log n}{n}$, the expected number of $k$-tons is $\sim\frac{(\log n)^{k-m+1}}{e^xk!}$, so that $\frac{k!S_k}{(\log n)^{k-m+1}}\sim e^{-x}$.
\newline
\newline
While Theorem \ref{thm3} provides us with the joint distribution between the total number of coupons collected and the number of $k$-tons for minimal $k$, we can also ask what happens for maximum values of $k$. The answer to this is provided in the next theorem. 
\begin{thm}\label{thm5}
Fix $m\in\mathbb{N}$ and let $T_m(n)$ be the total number of coupons collected to see at least $m$ copies of each coupon.
\begin{enumerate}
\item Fix $d\in\mathbb{R}$. Pick an increasing sequence $(n_j)_j\subset\mathbb{N}$ and a sequence $(d_j)_j\subset\mathbb{R}$ such that $\lim_{j\rightarrow\infty}d_j=d$ and $k_j=e\log n_j+\left((e-1)(m-1)-\frac{1}{2}\right)\log\log n_j+d_j\in\mathbb{N}$ for all $j\in\mathbb{N}$. Then the joint probability distribution of $\left(\frac{T_m\left(n_j\right)-n_j\log n_j-(m-1)n_j\log\log n_j}{n_j},S_{k_j}\right)$ converges to $\left(X,Pois\left(\frac{e^{(e-1)X-d}}{\sqrt{2\pi e}}\right)\right)$, where $X\sim\mathrm{Gumbel}(-\log((m-1)!),1)$.
\item Let $g(n)=o\left(\log\log n\right)$ with $\lim_{n\rightarrow\infty}g(n)=\infty$ such that
\begin{equation*}
k=e\log n+\left((e-1)(m-1)-\frac{1}{2}\right)\log\log n+g(n)\in\mathbb{N}
\end{equation*}
for all $n\in\mathbb{N}$. Then $S_k$ converges to $0$ with probability $1$.
\item Let $g(n)=o\left(\log\log n\right)$ with $\lim_{n\rightarrow\infty}g(n)=\infty$ such that
\begin{equation*}
k=e\log n+\left((e-1)(m-1)-\frac{1}{2}\right)\log\log n-g(n)\in\mathbb{N}
\end{equation*}
for all $n\in\mathbb{N}$. Then the value of $P\left(S_k=0\right)$ converges to $0$.
\end{enumerate}
\end{thm}
A simple heuristic argument for Theorem \ref{thm5} may be given as follows. The probability of a type of coupon being seen $k_j$ times is
\begin{equation*}
\binom{T_m\left(n_j\right)}{k_j}\left(\frac{1}{n_j}\right)^{k_j}\left(1-\frac{1}{n_j}\right)^{T_m\left(n_j\right)-k_j}.
\end{equation*}
Assuming that $x=\frac{T_m\left(n_j\right)-n_j\log n_j-(m-1)n_j\log\log n_j}{n_j}$ is constant as $j\rightarrow\infty$ and $k_j$ is as defined in Theorem \ref{thm5}, we have that the above is asymptotic to
\begin{align*}
\frac{1}{\sqrt{2\pi k_j}}\left(\frac{eT_m\left(n_j\right)}{n_jk_j}\right)^{k_j}e^{-\frac{T_m\left(n_j\right)}{n_j}}&\sim\frac{1}{\sqrt{2\pi e\log n_j}}e^{\frac{T_m\left(n_j\right)e-k_jn_j-T_m\left(n_j\right)}{n_j}}\\
&=\frac{1}{\sqrt{2\pi e\log n_j}}e^{-\log n_j+\frac{\log\log n_j}{2}+(e-1)x-d_j}\\
&=\frac{1}{\sqrt{2\pi e}n_j}e^{(e-1)x-d_j}.
\end{align*}
Therefore, by similar reasoning in the heuristic for Theorem \ref{thm3}, we can see that the probability that there are $y$ $k$-tons is asymptotic to
\begin{equation*}
\binom{n}{y}\left(\frac{\lambda}{n}\right)^y\left(1-\frac{\lambda}{n}\right)^{n-y}\sim\frac{\lambda^ye^{-\lambda}}{y!},
\end{equation*}
where $\frac{e^{(e-1)x-d}}{\sqrt{2\pi e}n_j}$, which gives the poisson distribution in Theorem \ref{thm5}, part 1). Parts 2) and 3) can be heuristically argued by letting $d_j$ approach $\infty$ or $-\infty$.
\newline
\newline
From Theorem \ref{thm3} we can derive the following corollaries.
\begin{cor}\label{thm1}
Fix $m\in\mathbb{N}$ and let $n\rightarrow\infty$ and let $k=o(\log n)$ if $m=1$ and $k=o\left(\frac{\log n}{\log\log n}\right)$ if $m\geq 2$, with $k\geq m$ in either case. Suppose we keep collecting coupons until we see at least $m$ copies of each coupon. Then $\frac{k!S_k}{(\log n)^{k-m+1}}$ converges in distribution to $\mathrm{Exp}\left(\frac{1}{(m-1)!}\right)$.
\end{cor}
Corollary \ref{thm1} gives the asymptotic probability distribution of the number $k$-tons. If we fix $k$, we can also determine the joint asymptotic distribution between the number of $m$-tons, $m+1$-tons,..., and $k$-tons. We can see that asymptotically all of these variables are linearly dependent.
\begin{cor}\label{thm2}
Fix $k,m\in\mathbb{N}$ with $k\geq m$ and let $n\rightarrow\infty$. Suppose we keep collecting coupons until we see at least $m$ copies of each coupon. Then the joint probability distribution $\left(\frac{m!S_m}{(\log n)},\frac{(m+1)!S_{m+1}}{(\log n)^2},\ldots,\frac{k!S_k}{(\log n)^{k-m+1}}\right)$ converges to $X\cdot(1,1,\ldots,1)$ where $X\sim\mathrm{Exp}\left(\frac{1}{(m-1)!}\right)$.
\end{cor}
\section{Proofs}
To prove Theorem \ref{thm3} we require the following notation and lemma.
\begin{note}
Suppose we collect $x$ coupons in total regardless if we have collected all $n$ types of coupons or not. Let $S_{i,x}$ denote the number of types of coupons we collected exactly $i$ times. 
\end{note}
\begin{note}
Throughout this paper, we will let $N(n,f(n)):=n\log n+(m-1)n\log\log n+f(n)$.
\end{note}
\begin{lemma}\label{lem4}
Let $f:\mathbb{N}\rightarrow\mathbb{R}$ be such that $-\frac{n\log\log n}{2}\leq f(n)\leq\frac{n\log\log n}{2}$, and $N(n,f(n))\in\mathbb{N}$ for all $n\in\mathbb{N}$. Suppose we collect $N(n,f(n))$ coupons in total regardless if we collect $m$ copies of all $n$ different kinds of coupons or not. Let $k=o(\log n)$ if $m=1$ and $k=o\left(\frac{\log n}{\log\log n}\right)$ if $m\geq 2$. If $m\leq k$, then as $n\rightarrow\infty$ $\frac{k!S_{k,N(n,f(n))}}{(\log n)^{k-m+1}}$ is asymptotic to $e^{-\frac{f(n)}{n}}$ with probability $1$.
\end{lemma}
\begin{proof}
First, assume that $m\leq k$. We can see that as $n\rightarrow\infty$
\begin{align}
E(S_{k,N(n,f(n))})&=n\cdot P(c_0\text{ occurs exactly }k\text{ times in the collection of }N(n,f(n))\text{ coupons})\nonumber\\
&=\binom{N(n,f(n))}{k}\left(1-\frac{1}{n}\right)^{N(n,f(n))-k}\frac{1}{n^{k-1}}\nonumber\\
&\sim\frac{N(n,f(n))^k}{k!n^{k-1}}e^{-\log n-(m-1)\log\log n-\frac{f(n)}{n}}\nonumber\\
&\sim\frac{(\log n)^{k-m+1}}{k!e^{\frac{f(n)}{n}}}.\label{eqn3}
\end{align}
Thus, we can deduce
\begin{equation}\label{eqn1}
E\left(\frac{k!S_{k,N(n,f(n))}}{(\log n)^{k-m+1}}\right)=e^{-\frac{f(n)}{n}}(1+o(1))
\end{equation}
as $n\rightarrow\infty$. Also,
\begin{align}
E(S_{k,N(n,f(n))}^2)&=n(n-1)\cdot P(c_0,c_1\text{ both occur exactly }k\text{ times in the collection of }N(n,f(n))\text{ coupons})\nonumber\\
&\quad+n\cdot P(c_0\text{ occurs exactly }k\text{ times in the collection of }N(n,f(n))\text{ coupons})\nonumber\\
&=n(n-1)\binom{N(n,f(n))}{k}\binom{N(n,f(n))-k}{k}\left(1-\frac{2}{n}\right)^{N(n,f(n))-2k}\frac{1}{n^{2k}}\nonumber\\
&\quad+\binom{N(n,f(n))}{k}\left(1-\frac{1}{n}\right)^{N(n,f(n))-k}\frac{1}{n^{k-1}}\nonumber\\
&\sim\frac{(\log n)^{2k-2m+2}}{(k!)^2e^{\frac{2f(n)}{n}}}+\frac{(\log n)^{k-m+1}}{k!e^{\frac{f(n)}{n}}}.\label{eqn2}
\end{align}
Notice that
\begin{equation}\label{eqn4}
\frac{k!e^{\frac{f(n)}{n}}}{(\log n)^{k-m+1}}\leq\frac{k!}{(\log n)^{k-m+\frac{1}{2}}}.
\end{equation}
Notice that if $k$ is bounded as $n\rightarrow\infty$, then the right-hand side of \eqref{eqn4} tends toward $0$. On the other hand for $k$ sufficiently large, we have $k!<k^{k-m}$, so even if values of $k$ tend toward $\infty$, the right-hand side of $\eqref{eqn4}$ tends toward $0$. Thus, from \eqref{eqn2}, we obtain
\begin{equation*}
E\left(\left(\frac{k!S_{k,N(n,f(n))}}{(\log n)^{k-m+1}}\right)^2\right)=e^{-\frac{2f(n)}{n}}(1+o(1))
\end{equation*}
as $n\rightarrow\infty$. Thus
\begin{equation*}
\mathrm{Var}\left(\frac{k!S_{k,N(n,f(n))}}{(\log n)^{k-m+1}}\right)=o(1)
\end{equation*}
as $n\rightarrow\infty$ so that we obtain our result for $m\leq k$.
\end{proof}
\begin{proof}[Proof of Theorem \ref{thm3}]
Erd\H{o}s and R\'{e}nyi \cite{erdos} have proved that $X=\frac{T_m(n)-n\log n-(m-1)n\log\log n}{n}$ is asymptotically Gumbel distributed. To prove Theorem \ref{thm3} it therefore suffices to show that
\begin{equation}\label{xy1}
\lim_{n\rightarrow\infty}P\left(\frac{T_m(n)-n\log n-(m-1)n\log\log n}{n}<x,\frac{k!S_k}{(\log n)^{k-m+1}}<e^{-y}\right)=0
\end{equation}
for any fixed $x<y$ and that
\begin{equation}\label{xy2}
\lim_{n\rightarrow\infty}P\left(\frac{T_m(n)-n\log n-(m-1)n\log\log n}{n}>x,\frac{k!S_k}{(\log n)^{k-m+1}}>e^{-y}\right)=0
\end{equation}
for any fixed $y<x$. We first prove \eqref{xy1}. To prove \eqref{xy1} we first calculate an upper bound for
\begin{equation*}
P\left(T_m(n)\leq N(n,nx),\frac{k!S_k}{(\log n)^{k-m+1}}<e^{-y}\right)
\end{equation*}
for all sufficiently large $n\in\mathbb{N}$. Let $M(n)=n\log n+\left(m-\frac{3}{2}\right)n\log\log n$. We have
\begin{align*}
&\quad P\left(T_m(n)\leq N(n,nx),\frac{k!S_k}{(\log n)^{k-m+1}}<e^{-y}\right)\\
&=P\left(T_m(n)< M(n),\frac{k!S_k}{(\log n)^{k-m+1}}<e^{-y}\right)+P\left(M(n)\leq T_m(n)\leq N(n,nx),\frac{k!S_k}{(\log n)^{k-m+1}}<e^{-y}\right)\\
&\leq P(T_m(n)< M(n))+P\left(M(n)\leq T_m(n)\leq N(n,nx),\frac{k!S_k}{(\log n)^{k-m+1}}<e^{-y}\right).
\end{align*}
Suppose it is true that $M(n)\leq T\leq N(n,nx)$ and $\frac{k!S_k}{(\log n)^{k-m+1}}<e^{-y}$. This implies that there exists some $M(n)\leq r\leq N(n,nx)$ such that after collecting the first $r$ coupons we have $\frac{k!S_{k,r}}{(\log n)^{k-m+1}}<e^{-y}$. Then we see that we can take $r=\lceil M(n)\rceil$. Therefore, we have
\begin{equation*}
P\left(T_m(n)\leq N(n,nx),\frac{k!S_k}{(\log n)^{k-m+1}}<e^{-y}\right)\leq P(T_m(n)< M(n))+P\left(\frac{k!S_{k,\lceil M(n)\rceil}}{(\log n)^{k-m+1}}<e^{-y}\right).
\end{equation*}
Now we let $n\rightarrow\infty$. From Erd\H{o}s and R\'{e}nyi's result that $X=\frac{T_m(n)-n\log n-(m-1)n\log\log n}{n}$ is asymptotically Gumbel distributed, we can deduce that
\begin{equation*}
\lim_{n\rightarrow\infty}P(T< M(n))=0.
\end{equation*}
From Lemma \ref{lem4}, we can deduce that as $n\rightarrow\infty$ $\frac{k!S_{k,\lceil M(n)\rceil}}{(\log n)^{k-m+1}}$ is asymptotic to $e^{\frac{\log\log n}{2}}=\sqrt{\log n}$ with probability $1$. Therefore,
\begin{equation*}
\lim_{n\rightarrow\infty}P\left(\frac{k!S_{k,\lceil M(n)\rceil}}{(\log n)^{k-m+1}}<e^{-y}\right)=0.
\end{equation*}
Thus, we obtain \eqref{xy1}. \eqref{xy2} is proved similarly, replacing $M(n)=n\log n+\left(m-\frac{3}{2}\right)n\log\log n$ with $M'(n)=n\log n+\left(m-\frac{1}{2}\right)n\log\log n$ and reversing the appropriate inequalities.
\end{proof}
\begin{proof}[Proof of Corollaries \ref{thm1} and \ref{thm2}]
Corollaries \ref{thm1} and \ref{thm2} follow from Theorem \ref{thm3} by showing that, if $X\sim\mathrm{Gumbel}(-\log((m-1)!),1)$, then $e^{-X}\sim\mathrm{Exp}\left(\frac{1}{(m-1)!}\right)$. Indeed, we have the following for any $r>0$:
\begin{align*}
P\left(e^{-X}\leq r\right)&=P(X\geq -\log r)\\
&=1-P(X\leq -\log r)\\
&=1-e^{-e^{-(-\log r+\log(m-1)!)}}\\
&=1-e^{-\frac{r}{(m-1)!}}.
\end{align*}
Hence, the result follows.
\end{proof}
To prove Theorem \ref{thm5}, we require the following lemma.
\begin{lemma}\label{lem1}
Fix $m\in\mathbb{N}$. Let $f:\mathbb{N}\rightarrow\mathbb{R}$ be such that $f(n)=O(n)$ and $N(n,f(n))\in\mathbb{N}$ for all $n\in\mathbb{N}$. Suppose we collect $N(n,f(n))$ coupons in total regardless if we collect $m$ copies of all $n$ different kinds of coupons or not. Pick an increasing sequence $(n_j)_j\subset\mathbb{N}$ and a sequence $(d_j)_j\subset\mathbb{R}$ such that $\lim_{j\rightarrow\infty}d_j=d$ for some $d\in\mathbb{R}$ and $k=e\log n_j+\left((e-1)(m-1)-\frac{1}{2}\right)\log\log n_j+d_j\in\mathbb{N}$ for all $j\in\mathbb{N}$. Then as $j\rightarrow\infty$ we have
\begin{align*}
&\quad P\left(S_{m-1,N\left(n_j,f\left(n_j\right)\right)}=r_1,S_{k,N\left(n_j,f\left(n_j\right)\right)}=r_2\right)\\
&=\frac{1}{r_1!r_2!}\left(\frac{e^{-\frac{f\left(n_j\right)}{n_j}}}{(m-1)!}\right)^{r_1}e^{\frac{-e^{-\frac{f\left(n_j\right)}{n_j}}}{(m-1)!}}\left(\frac{e^{\frac{(e-1)f\left(n_j\right)}{n_j}-d}}{\sqrt{2\pi e}}\right)^{r_2}e^{\frac{-e^{\frac{(e-1)f\left(n_j\right)}{n_j}-d}}{\sqrt{2\pi e}}}\left(1+O\left(\frac{\left(\log\log n_j\right)^2}{\log n_j}\right)\right),
\end{align*}
where the implied constant in the error term only depends upon $r_1$ and $r_2$.
\end{lemma}
\begin{proof}
Take some $j\in\mathbb{N}$. Let the probability that $r_1$ prescribed types of coupons are $(m-1)$-tons and $r_2$ prescribed types of coupons are $k$-tons be denoted as $W_{r_1,r_2}(n)$. Then, as $j\rightarrow\infty$, we have
\begin{align}
&\quad\binom{n_j}{r_1,r_2,n_j-r_1-r_2}W_{r_1,r_2}(n_j)\nonumber\\
&=\binom{n_j}{r_1,r_2,n-r_1-r_2}\frac{N\left(n_j,f\left(n_j\right)\right)!}{(m-1)!^{r_1}k!^{r_2}\left(N(n_j)-(m-1)r_1-kr_2\right)!n_j^{(m-1)r_1+kr_2}}\left(1-\frac{r_1+r_2}{n_j}\right)^{N(n_j)-(m-1)r_1-kr_2}\nonumber\\
&=\binom{n_j}{r_1,r_2,n_j-r_1-r_2}\frac{N\left(n_j,f\left(n_j\right)\right)!e^{-(r_1+r_2)N(n_j)/n_j}}{(m-1)!^{r_1}k!^{r_2}\left(N(n_j)-(m-1)r_1-kr_2\right)!n_j^{(m-1)r_1+kr_2}}\left(1+O\left(\frac{\log n_j}{n_j}\right)\right)\nonumber\\
&=\frac{N\left(n_j,f\left(n_j\right)\right)^{(m-1)r_1+kr_2}e^{-\frac{\left(r_1+r_2\right)f\left(n_j\right)}{n_j}}}{r_1!r_2!(m-1)!^{r_1}k!^{r_2}\left(\log n_j\right)^{\left(r_1+r_2\right)(m-1)}n_j^{(m-1)r_1+kr_2}}\left(1+O\left(\frac{\log n_j}{n_j}\right)\right)\nonumber\\
&=\frac{\left(\log n_j\right)^{kr_2+(e-1)(m-1)r_2}e^{-\frac{\left(r_1+r_2(e+1)\right)f\left(n_j\right)}{n_j}}}{r_1!r_2!(m-1)!^{r_1}k!^{r_2}}\left(1+O\left(\frac{\left(\log\log n_j\right)^2}{\log n_j}\right)\right)\nonumber\\
&=\frac{1}{r_1!r_2!}\left(\frac{e^{-\frac{f\left(n_j\right)}{n_j}}}{(m-1)!}\right)^{r_1}\left(\frac{e^{\frac{(e-1)f\left(n_j\right)}{n_j}-d}}{\sqrt{2\pi e}}\right)\left(1+O\left(\frac{\left(\log\log n_j\right)^2}{\log n_j}\right)\right),\label{eqn5}
\end{align}
where the implied constant in the error term only depends on $r_1$ and $r_2$. By the inclusion-exclusion formula, we have
\begin{align}
&\quad P\left(S_{m-1,N\left(n_j,f\left(n_j\right)\right)}=r_1,S_{k,N\left(n_j,f\left(n_j\right)\right)}=r_2\right)\nonumber\\
&=\sum_{0\leq j_1+j_2\leq n_j-r_1-r_2}(-1)^{j_1+j_2}\binom{n_j}{r_1+j_1,r_2+j_2,n_j-r_1-r_2-j_1-j_2}\binom{r_1+j_1}{r_1}\binom{r_2+j_2}{r_2}W_{r_1+j_1,r_2+j_2}\left(n_j\right).\label{eqn6}
\end{align}
For $t\leq n_j-r_1-r_2$ even we have
\begin{align}
&\quad P\left(S_{m-1,N\left(n_j,f\left(n_j\right)\right)}=r_1,S_{k,N\left(n_j,f\left(n_j\right)\right)}=r_2\right)\nonumber\\
&\leq\sum_{0\leq j_1+j_2\leq t}(-1)^{j_1+j_2}\binom{n_j}{r_1+j_1,r_2+j_2,n_j-r_1-r_2-j_1-j_2}\binom{r_1+j_1}{r_1}\binom{r_2+j_2}{r_2}W_{r_1+j_1,r_2+j_2}\left(n_j\right)\label{eqn7}
\end{align}
and for $t\leq n_j-r_1-r_2$ odd we have
\begin{align}
&\quad P\left(S_{m-1,N\left(n_j,f\left(n_j\right)\right)}=r_1,S_{k,N\left(n_j,f\left(n_j\right)\right)}=r_2\right)\nonumber\\
&\geq\sum_{0\leq j_1+j_2\leq t}(-1)^{j_1+j_2}\binom{n_j}{r_1+j_1,r_2+j_2,n_j-r_1-r_2-j_1-j_2}\binom{r_1+j_1}{r_1}\binom{r_2+j_2}{r_2}W_{r_1+j_1,r_2+j_2}\left(n_j\right).\label{eqn8}
\end{align}
Combining \eqref{eqn6}, \eqref{eqn7}, and \eqref{eqn8}, we have
\begin{align*}
&\quad P\left(S_{m-1,N\left(n_j,f\left(n_j\right)\right)}=r_1,S_{k,N\left(n_j,f\left(n_j\right)\right)}=r_2\right)\\
&=\frac{1}{r_1!r_2!}\sum_{0\leq j_1+j_2\leq n_j}\frac{(-1)^{j_1+j_2}}{j_1!j_2!}\left(\frac{e^{-\frac{f\left(n_j\right)}{n_j}}}{(m-1)!}\right)^{r_1+j_1}\left(\frac{e^{\frac{(e-1)f\left(n_j\right)}{n_j}-d}}{\sqrt{2\pi e}}\right)^{r_2+j_2}\left(1+O\left(\frac{\left(\log\log n_j\right)^2}{\log n_j}\right)\right)\\
&=\frac{1}{r_1!r_2!}\left(\frac{e^{-\frac{f\left(n_j\right)}{n_j}}}{(m-1)!}\right)^{r_1}e^{\frac{-e^{-\frac{f\left(n_j\right)}{n_j}}}{(m-1)!}}\left(\frac{e^{\frac{(e-1)f\left(n_j\right)}{n_j}-d}}{\sqrt{2\pi e}}\right)^{r_2}e^{\frac{-e^{\frac{(e-1)f\left(n_j\right)}{n_j}-d}}{\sqrt{2\pi e}}}\left(1+O\left(\frac{\left(\log\log n_j\right)^2}{\log n_j}\right)\right),
\end{align*}
where again the implied constant in the error term only depends on $r_1$ and $r_2$.
\end{proof}
\begin{proof}[Proof of Theorem \ref{thm5}]
We first prove the first part of Theorem \ref{thm5}. Let $d\in\mathbb{R}$ and pick increasing sequences $\left(n_j\right)_j$ and $\left(d_j\right)_j$ such that $\lim_{j\rightarrow\infty}d_j=d$ and $k_j:=e\log n_j+\left((e-1)(m-1)-\frac{1}{2}\right)\log\log n_j+d_j\in\mathbb{N}$ for all $j\in\mathbb{N}$. Let $x<y$ and $r\in\mathbb{N}$ be constants. Consider the probability
\begin{equation*}
P\left(N\left(n_j,n_jx\right)\leq T_m(n)\leq N\left(n_j,n_jy\right),S_k=r\right).
\end{equation*}
We have
\begin{equation*}
P\left(N\left(n_j,n_jx\right)\leq T_m(n)\leq N\left(n_j,n_jy\right),S_k=r\right)=\sum_{M=N\left(n_j,n_jx\right)}^{N\left(n_j,n_jy\right)}P\left(T=M,S_{k,M-1}=r\right).
\end{equation*}
Clearly,
\begin{equation}\label{eqn9}
P\left(T_m(n)=M,S_{k,M-1}=r\right)\leq P\left(S_{m-1,M-1}=1,S_{m-1,M}=0,S_{k,M-1}=r\right)=\frac{P\left(S_{m-1,M-1}=1,S_{k,M-1}=r\right)}{n}.
\end{equation}
Also, for $M\geq N\left(n_j,n_jx\right)$, if we have $S_{i,M}>0$ for some $0\leq i\leq m-2$, then we have $S_{j,N\left(n_l,n_lx\right)}>0$ for some $0\leq l\leq i$. Thus, we can also deduce
\begin{align}
P\left(N\left(n_j,n_jx\right)\leq T_m(n)\leq N\left(n_j,n_jy\right),S_k=r\right)\geq &\left(\sum_{M=N\left(n_j,n_jx\right)}^{N\left(n_j,n_jy\right)}\frac{P\left(S_{m-1,M-1}=1,S_{k,M-1}=r\right)}{n_j}\right)\nonumber\\
&-P\left(\exists\text{ }0\leq l\leq m-2\text{ }S_{l,N\left(n_j,n_jx\right)}>0\right).\label{eqn10}
\end{align}
For any $0\leq l\leq m-2$ notice that as $j\rightarrow\infty$ we have
\begin{align*}
E\left(S_{l,N\left(n_j,n_jx\right)}\right)&=n\cdot P(c_0\text{ occurs exactly }l\text{ times in the collection of }N\left(n_j,n_jx\right)\text{ coupons})\\
&=\binom{N\left(n_j,n_jx\right)}{l}\left(1-\frac{1}{n_j}\right)^{N\left(n_j,n_jx\right)-l}\frac{1}{n_j^{l-1}}\\
&\sim\frac{N\left(n_j,n_jx\right)^l}{l!n_j^{l-1}}e^{-\log n_j-(m-1)\log\log n_j-x}\\
&\sim\frac{(\log n_j)^{l-m+1}e^{-x}}{l!}.
\end{align*}
Since $l\leq m-2$, we thus have
\begin{equation*}
\lim_{j\rightarrow\infty}E\left(S_{l,N\left(n_j,n_jx\right)}\right)=0,
\end{equation*}
and so
\begin{equation}\label{eqn11}
\lim_{j\rightarrow\infty}P\left(\exists\text{ }0\leq l\leq m-2\text{ }S_{l,N\left(n_j,n_jx\right)}>0\right)=0.
\end{equation}
For every $N\left(n_j,n_jx\right)\leq M\leq N\left(n_j,n_jy\right)$, let $c_{M,j}:=\frac{M-n_j\log n_j-(m-1)n_j\log\log n_j}{n_j}$.
By Lemma \ref{lem1}, we have
\begin{align}
&\quad\sum_{M=N\left(n_j,n_jx\right)}^{N\left(n_j,n_jy\right)}\frac{P\left(S_{m-1,M-1}=1,S_{k,M-1}=r\right)}{n}\nonumber\\
&=\sum_{M=N\left(n_j,n_jx\right)}^{N\left(n_j,n_jy\right)}\frac{1}{r!n}\left(\frac{e^{-c_{M,j}}}{(m-1)!}\right)e^{\frac{-e^{-c_{M,j}}}{(m-1)!}}\left(\frac{e^{(e-1)c_{M,j}-d}}{\sqrt{2\pi e}}\right)^re^{\frac{-e^{(e-1)c_{M,j}-d}}{\sqrt{2\pi e}}}\left(1+O\left(\frac{\left(\log\log n_j\right)^2}{\log n_j}\right)\right)\nonumber\\
&<\left(\frac{e^{(e-1)y-d}}{\sqrt{2\pi e}}\right)^re^{\frac{-e^{(e-1)x-d}}{\sqrt{2\pi e}}}\sum_{M=N\left(n_j,n_jx\right)}^{N\left(n_j,n_jy\right)}\frac{1}{r!n}\left(\frac{e^{-c_{M,j}}}{(m-1)!}\right)e^{\frac{-e^{-c_{M,j}}}{(m-1)!}}\left(1+O\left(\frac{\left(\log\log n_j\right)^2}{\log n_j}\right)\right).\label{eqn12}
\end{align}
Note that we have
\begin{align}
\lim_{j\rightarrow\infty}\sum_{M=N\left(n_j,n_jx\right)}^{N\left(n_j,n_jy\right)}\frac{1}{n}\left(\frac{e^{-c_{M,j}}}{(m-1)!}\right)e^{\frac{-e^{-c_{M,j}}}{(m-1)!}}&=\int_x^y\frac{e^{-t}e^{\frac{-e^{-t}}{(m-1)!}}}{(m-1)!}dt\nonumber\\
&=e^{\frac{-e^{-y}}{(m-1)!}}-e^{\frac{-e^{-x}}{(m-1)!}}\nonumber\\
&=\lim_{j\rightarrow\infty}P\left(N\left(n_j,x\right)\leq T_m\left(n_j\right)\leq N\left(n_j,y\right)\right).\label{eqn13}
\end{align}
Combining \eqref{eqn9}, \eqref{eqn10}, \eqref{eqn11}, \eqref{eqn12}, and \eqref{eqn13}, we obtain
\begin{align*}
&\quad\limsup_{j\rightarrow\infty}P\left(N\left(n_j,n_jx\right)\leq T_m\left(n_j\right)\leq N\left(n_j,n_jy\right),S_k=r\right)\\
&\leq\frac{1}{r!}\left(\frac{e^{(e-1)y-d}}{\sqrt{2\pi e}}\right)^re^{\frac{-e^{(e-1)x-d}}{\sqrt{2\pi e}}}\lim_{j\rightarrow\infty}P\left(N\left(n_j,n_jx\right)\leq T_m\left(n_j\right)\leq N\left(n_j,n_jy\right)\right).
\end{align*}
By similar reasoning, we can also obtain
\begin{align*}
&\quad\liminf_{j\rightarrow\infty}P\left(N\left(n_j,n_jx\right)\leq T_m\left(n_j\right)\leq N\left(n_j,n_jy\right),S_k=r\right)\\
&\geq\frac{1}{r!}\left(\frac{e^{(e-1)x-d}}{\sqrt{2\pi e}}\right)^re^{\frac{-e^{(e-1)y-d}}{\sqrt{2\pi e}}}\lim_{j\rightarrow\infty}P\left(N\left(n_j,n_jx\right)\leq T_m\left(n_j\right)\leq N\left(n_j,n_jy\right)\right).
\end{align*}
Since the choices of $x<y$ were arbitrary, we can see that the joint limiting distribution holds.
\newline
\newline
For the second part, again let $x<y$ be constants. We will show that
\begin{equation*}
\lim_{n\rightarrow\infty}P\left(N\left(n,nx\right)\leq T_m\left(n_j\right)\leq N\left(n,ny\right),S_k\geq 1\right)=0,
\end{equation*}
which implies the desired result since $x<y$ are arbitrary. We have
\begin{align*}
&\quad P\left(N\left(n,nx\right)\leq T_m(n)\leq N\left(n,ny\right),S_k\geq 1\right)\\
&\leq\sum_{M=N\left(n,x\right)}^{N\left(n,y\right)}\frac{P\left(S_{m-1,M-1}=1,S_{k,M-1}\geq 1\right)}{n}\\
&\leq\sum_{M=N\left(n,nx\right)}^{N\left(n,ny\right)}\frac{P\left(S_{k,M-1}\geq 1\right)}{n}\\
&\leq\sum_{M=N\left(n,nx\right)}^{N\left(n,ny\right)}\frac{n\cdot P\left(c_0\text{ occurs exactly }k\text{ times in the collection of }M\text{ coupons}\right)}{n}\\
&=\sum_{M=N\left(n,nx\right)}^{N\left(n,ny\right)}P\left(c_0\text{ occurs exactly }k\text{ times in the collection of }M\text{ coupons}\right)\\
&=\sum_{M=N\left(n,nx\right)}^{N\left(n,ny\right)}\binom{M}{k}\left(1-\frac{1}{n}\right)^{M-k}\frac{1}{n^k}\\
&<\frac{1}{k!}\sum_{M=N\left(n,nx\right)}^{N\left(n,ny\right)}\left(1-\frac{1}{n}\right)^{M-k}\frac{M^k}{n^k}\\
&<\frac{n(y-x)}{k!}\left(1-\frac{1}{n}\right)^{N(n,nx)-k}(\log n+(m-1)\log\log n+y)^k.
\end{align*}
As $n\rightarrow\infty$ we have
\begin{align}
&\quad\frac{n(y-x)}{k!}\left(1-\frac{1}{n}\right)^{N(n,nx)-k}(\log n+(m-1)\log\log n+y)^k\nonumber\\
&\sim\frac{(y-x)(\log n+(m-1)\log\log n+y)^k}{k!(\log n)^{m-1}e^x}\nonumber\\
&\sim\frac{y-x}{\sqrt{2\pi e}(\log n)^{m-\frac{1}{2}}e^x}\left(\frac{e\log n+e(m-1)\log\log n+ey}{k}\right)^k\nonumber\\
&\sim\frac{(y-x)e^{\left(m-\frac{1}{2}\right)\log\log n+ey-g(n)}}{\sqrt{2\pi e}(\log n)^{m-\frac{1}{2}}e^x}\nonumber\\
&=\frac{(y-x)e^{ey-x-g(n)}}{\sqrt{2\pi e}}.\label{eqn14}
\end{align}
We have
\begin{equation*}
\lim_{n\rightarrow\infty}\frac{(y-x)e^{ey-x-g(n)}}{\sqrt{2\pi e}}=0,
\end{equation*}
giving us our result.
\newline
\newline
For the third part, again let $x<y$ be constants. We will show that
\begin{equation*}
\lim_{n\rightarrow\infty}P\left(N\left(n,nx\right)\leq T_m(n)\leq N\left(n,ny\right),S_k=0\right)=0,
\end{equation*}
which implies the desired result since $x<y$ are arbitrary. We have
\begin{align}
P\left(N\left(n,nx\right)\leq T_m(n)\leq N\left(n,ny\right),S_k=0\right)&\leq\sum_{M=N\left(n,nx\right)}^{N\left(n,ny\right)}\frac{P\left(S_{m-1,M-1}=1,S_{k,M-1}=0\right)}{n}\nonumber\\
&\leq\sum_{M=N\left(n,nx\right)}^{N\left(n,ny\right)}\frac{P\left(S_{k,M-1}=0\right)}{n}.\label{eqn15}
\end{align}
Note that we have
\begin{equation*}
E\left(S_{k,M}\right)=\binom{M}{k}\left(1-\frac{1}{n}\right)^{M-k}\frac{1}{n^{k-1}}
\end{equation*}
and
\begin{align*}
E\left(S_{k,M}^2\right)&=\binom{M}{k}\binom{M-k}{k}\left(1-\frac{2}{n}\right)^{M-2k}\frac{(n-1)}{n^{2k-1}}+\binom{M}{k}\left(1-\frac{1}{n}\right)^{M-k}\frac{1}{n^{k-1}}\\
&<\binom{M}{k}^2\left(1-\frac{1}{n}\right)^{2M-4k}\frac{1}{n^{2k}}+\binom{M}{k}\left(1-\frac{1}{n}\right)^{M-k}\frac{1}{n^{k-1}}.
\end{align*}
We have
\begin{align*}
Var\left(S_{k,M}\right)=E\left(S_{k,M}^2\right)-E\left(S_{k,M}\right)^2\leq E\left(S_{k,M}\right)^2\left(\left(1-\frac{1}{n}\right)^{-2k}-1\right)+E\left(S_{k,M}\right).
\end{align*}
Since $k<n$ for sufficiently large $n$, we have that there exists $C_1>0$ such that for sufficiently large $n$, we have
\begin{equation*}
Var\left(S_{k,M}\right)\leq\frac{C_1kE\left(S_{k,M}\right)^2}{n}+E\left(S_{k,M}\right).
\end{equation*}
Similarly to how we derived \eqref{eqn14}, we can deduce that there exists $C_2>0$ depending only on $x$ and $y$ such that for sufficiently large $n$, we have $E\left(S_{k,M}\right)<C_2e^{g(n)}<C_2\log n$. Therefore, for sufficiently large $n$, we have
\begin{equation*}
Var\left(S_{k,M}\right)\leq E\left(S_{k,M}\right)+1<2E\left(S_{k,M}\right).
\end{equation*}
Thus, for sufficiently large $n$,
\begin{equation*}
\sigma\left(S_{k,M}\right)<\sqrt{2E\left(S_{k,M}\right)}.
\end{equation*}
Thus, by Chevyshev's Inequality, we have
\begin{align}
P\left(S_{k,M}=0\right)&\leq P\left(\left|S_{k,M}-E\left(S_{k,M}\right)\right|\geq E\left(S_{k,M}\right)\right)\nonumber\\
&\leq P\left(\left|S_{k,M}-E\left(S_{k,M}\right)\right|\geq\frac{\sigma\left(S_{k,M}\right)\sqrt{E\left(S_{k,M}\right)}}{\sqrt{2}}\right)\nonumber\\
&\leq\frac{2}{E\left(S_{k,M}\right)}\label{eqn16}
\end{align}
for sufficiently large $n$. As well, similarly to how we derived \eqref{eqn14}, we have that there exists $C_3>0$ depending on only $x$ and $y$ such that $C_3e^{g(n)}<E\left(S_{k,M}\right)$. Thus, from \eqref{eqn15} and \eqref{eqn16}, we have
\begin{equation*}
P\left(N\left(n,nx\right)\leq T_m(n)\leq N\left(n,ny\right),S_k=0\right)\leq\sum_{M=N\left(n,nx\right)}^{N\left(n,ny\right)}\frac{2}{C_3ne^{g(n)}}=\frac{2(y-x)}{C_3e^{g(n)}}
\end{equation*}
for sufficiently large $n$. Since $\lim_{n\rightarrow\infty}g(n)=\infty$, we have our result.
\end{proof}
\section{Future Work}
There are a few open questions that are worth exploring with respect to the coupon collector's problem and the number of $k$-tons. For instance, can we extend our ranges for $k$? We have results for $k=o(\log n)$ if $m=1$ and $k=o\left(\frac{\log n}{\log\log n}\right)$ if $m\geq 2$, as well as $k=e\log n+\left((e-1)(m-1)-\frac{1}{2}\right)\log\log n+d$, but we can also ask similar questions for values of $k$ between these minimum and maximum values. For instance, we can consider $k=\theta(\log n)$ or $k=\theta\left(\frac{\log n}{\log\log n}\right)$. We can also ask what happens if we have $m$ increasing with $n$. Also, Doumas and Papanicolaou \cite{doumas} studied the limiting distribution of $T_m(n)$ in the case of unequal coupon probabilities. We could also study the limiting distribution of $k$-tons in the case of unequal coupon probabilities.
\section{Acknowledgements}
The research of J.C. Saunders is supported by an Azrieli International Postdoctoral Fellowship, as well as a Postdoctoral Fellowship at the University of Calgary. The author would also like to thank Dr. Daniel Berend for helpful comments on the paper.

\end{document}